\documentclass[11pt]{article}
\usepackage{graphicx}
\usepackage{amsthm, amsmath, amssymb, tikz, bm}
\usepackage{mathtools}
\usepackage{xifthen}
\usepackage{comment}

\usepackage[margin=1.1in]{geometry} 

\usepackage[shortlabels]{enumitem}
\usepackage{todonotes}
\usepackage[colorlinks=true,
linkcolor=blue,citecolor=blue,
urlcolor=blue]{hyperref}

\usetikzlibrary{calc,shapes, backgrounds}

\newtheorem{theorem}{Theorem}[section]
\newtheorem{lemma}[theorem]{Lemma}

\newtheorem{claim}{Claim}

\newtheorem{question}{Question}

\newcommand{\dss}{\displaystyle\sum}

\newcommand{\lp}{\left (}
\newcommand{\rp}{\right )}

\newcommand{\cH}{\mathcal{H}}

\DeclarePairedDelimiter{\ceil}{\lceil}{\rceil}

\makeatletter
               {\list{}{\leftmargin=0pt % <------- Adjust this length
                        \labelwidth\z@ \itemindent-\leftmargin
                        }}%
               {\endlist}
\makeatother

\title{Polynomial $\chi$-binding functions for $t$-broom-free graphs}

\author{
Xiaonan Liu \thanks{Georgia Institute of Technology, Atlanta, GA 30332,
({\tt xliu729@gatech.edu}). This author was partially supported by NSF Grant DMS-1856645.}
\and
Joshua Schroeder \thanks{Georgia Institute of Technology, Atlanta, GA 30332,
({\tt jschroeder35@gatech.edu}). This author was partially supported by NSF Grant DMS-1856645.}
\and
Zhiyu Wang \thanks{Georgia Institute of Technology, Atlanta, GA 30332,
({\tt zwang672@gatech.edu}).}
\and
Xingxing Yu \thanks{Georgia Institute of Technology, Atlanta, GA 30332,
({\tt yu@math.gatech.edu}). This author was partially supported by NSF Grant DMS-1954134.}
}

\begin{document}

\maketitle

\begin{abstract}
 For any positive integer $t$, a \emph{$t$-broom} is a graph obtained from $K_{1,t+1}$ by subdividing an edge once.  In this paper, we show that, for graphs $G$ without induced $t$-brooms, we have $\chi(G) =  o(\omega(G)^{t+1})$, where  $\chi(G)$ and $\omega(G)$ are the chromatic number and clique number of $G$, respectively. When $t=2$, this answers a question of  Schiermeyer and Randerath. Moreover, for $t=2$, we strengthen the bound on $\chi(G)$ to $7\omega(G)^2$, confirming a conjecture of Sivaraman.
For $t\geq 3$ and \{$t$-broom, $K_{t,t}$\}-free graphs, we improve the bound to $o(\omega^{t})$.
\end{abstract}

\section{Introduction}\label{sec:intro}

A class of graphs $\mathcal{G}$ is called \textit{hereditary} if every induced subgraph of any graph in $\mathcal{G}$ also belongs to $\mathcal{G}$. One important and well-studied hereditary graph class is the family of \textit{$H$-free} graphs, i.e., graphs that have no induced subgraph isomorphic to a fixed graph $H$. Given a class of graphs $\cH$, we say that a graph $G$ is \textit{$\cH$-free} if $G$ is $H$-free for every $H\in \cH$.

For a graph $G$, $\chi(G)$ and $\omega(G)$ denote the chromatic number and the clique number of $G$, respectively. 
Tutte \cite{De47, De54} showed that for any $n$, there exists a triangle-free graph with chromatic number at least $n$. (See \cite{Mycielski1955} for another construction, and see \cite{Scott-Seymour2020} for more constructions.) Hence in general there exists no function of $\omega(G)$ that gives an upper bound on $\chi(G)$ for all graphs $G$.
A hereditary class of graphs $\mathcal{G}$ is called \textit{$\chi$-bounded} if there is a function $f$ (called a \textit{$\chi$-binding function}) such that $\chi(G) \leq f(\omega(G))$ for every $G\in \mathcal{G}$. If $f$ is additionally a polynomial function, then we say that $\mathcal{G}$ is \textit{polynomially $\chi$-bounded}. Graph classes with polynomial $\chi$-binding functions are important, as they satisfy the Erd\H{o}s-Hajnal conjecture \cite{EH} on the Ramsey number of $H$-free graphs. 

One well-known hereditary $\chi$-bounded graph class is the class of \textit{perfect graphs} (i.e., graphs $G$ such that every induced subgraph $H$ of $G$ satisfies $\chi(H)=\omega(H)$), which is a class of graphs for which the identity function is a $\chi$-binding function.
A \textit{hole} in a graph $G$ is an induced cycle in $G$ of length at least four. An \textit{antihole} of $G$ is an induced subgraph of $G$ whose complement graph is a cycle. 
Chudnovsky, Robertson, Seymour, and Thomas \cite{CRST2006} characterized perfect graphs as the set of $\{\text{odd hole}, \text{odd antihole}\}$-free graphs, known as the Strong Perfect Graph Theorem.  

One important research direction in the area of $\chi$-boundedness is about determining graph families $\cH$ such that the class of $\cH$-free graphs is $\chi$-bounded, as well as finding the smallest possible $\chi$-binding function for such hereditary class of graphs. By a probabilistic construction of Erd\H{o}s \cite{Erdos1959}, if $\cH$ is finite and none of the graphs in $\cH$ is acyclic, then the family of $\cH$-free graphs is not $\chi$-bounded. Gy\'arf\'as \cite{Gyarfas1975} and Sumner \cite{Sumner1981} independently conjectured that for every tree $T$, the class of $T$-free graphs is $\chi$-bounded. 
This conjecture has been confirmed for some special trees (see, for example, \cite{CSS2019, Gyarfas1975, GST1980, Kierstead-Penrice1994, Kierstead-Zhu2004, Scott1997, Scott-Seymour2020}), but remains open in general.

There is a natural connection between $\chi$-boundedness and the classical Ramsey number $R(m,n)$,  the smallest integer $N$ such that every graph on at least $N$ vertices contains an independent set on $m$ vertices or a clique on $n$ vertices.
Gy\'arf\'as \cite{Gyarfas1975} showed that the class of $K_{1,t}$-free graphs is $\chi$-bounded with the smallest $\chi$-binding function $f^*(\omega)$ satisfying $\frac{R(t,\omega+1)-1}{t-1} \leq f^*(\omega) \leq R(t,\omega)$. 
It is shown in \cite{AKS1980, Bohman2009, Bohman-Keevash2010, Kim1995} that $R(3,n) = \Theta(n^2/\log n)$ and, for fixed $t > 3$,
$c_1 (\frac{n}{\log n})^{\frac{t+1}{2}} \leq R(t,n) \leq c_2 \frac{n^{t-1}}{\log^{t-2} n}$, where $c_1$ and $c_2$ are absolute constants.
Hence for $K_{1,3}$-free (also known as \textit{claw-free}) graphs $G$, we have $\chi(G) = O(\omega(G)^2/\log \omega(G))$. Chudnovsky and Seymour \cite{Chudnovsky-Seymour2005} showed that if $G$ is a connected claw-free graph with independence number $\alpha(G)\ge 3$ then $\chi(G) \leq 2\omega(G)$.

In this paper, we consider a slightly larger class of graphs. For a positive integer $t$, a {\it $t$-broom} is the graph obtained from $K_{1,t+1}$ by subdividing an edge once. See the graph on the right in Figure \ref{fig:broom-graph}, and we denote that $t$-broom by $(u_0, v_1 v_2, u_1, u_2, \ldots, u_t)$ or by $(u_0,v_1 v_2,S)$, where $S=\{u_1, \ldots, u_t\}$.  (Note  that $u_0$ is the vertex of degree $t+1$, $v_1$ is the vertex of degree 2 whose neighbors are $u_0$ and $v_2$, and $u_1, \ldots, u_t$ are the remaining neighbors of $u_0$.)  

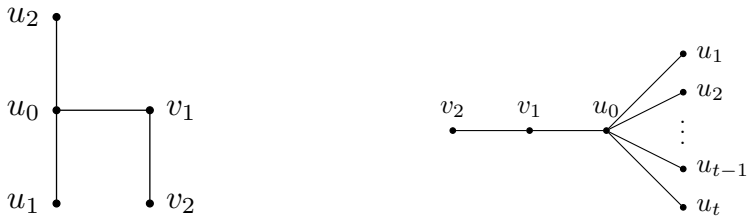
\begin{figure}[htb]
\hbox to \hsize{
	\hfil
	\resizebox{3cm}{!}{\begin{tikzpicture}[scale=1, Wvertex/.style={circle, draw=black, fill=white, scale=2}, bvertex/.style={circle, draw=black, fill=black, scale=0.2},rvertex/.style={circle, draw=red, fill=red, scale=0.2}]

\node [bvertex, label={[font=\small] right:$v_2$}] (v2) at (1,-1) {};
\node [bvertex, label={[font=\small] right:$v_1$}] (v1) at (1,0) {};
\node [bvertex, label={[font=\small] left:$u_0$}] (u0) at (0,0)  {};
\node [bvertex, label={[font=\small] left:$u_1$}] (u1) at (0,-1) {};
\node [bvertex, label={[font=\small] left:$u_2$}] (u2) at (0,1) {};

\draw (v2) -- (v1);
\draw (v1) -- (u0);
\draw (u0) -- (u1);
\draw (u0) -- (u2);

\end{tikzpicture}	}%
	\hfil
	\resizebox{4.5cm}{!}{\begin{tikzpicture}[scale=1, Wvertex/.style={circle, draw=black, fill=white, scale=2}, bvertex/.style={circle, draw=black, fill=black, scale=0.2},rvertex/.style={circle, draw=red, fill=red, scale=0.2}]

\node [bvertex, label={[font=\small] above:$v_2$}] (v2) at (-2,0) {};
\node [bvertex, label={[font=\small] above:$v_1$}] (v1) at (-1,0) {};
\node [bvertex, label={[font=\small] above:$u_0$}] (u0) at (0,0)  {};
\node [bvertex, label={[font=\small] right:$u_1$}] (u1) at (1,1) {};
\node [bvertex, label={[font=\small] right:$u_2$}] (u2) at (1,0.5) {};
\node [font=\scriptsize] at (1,0) {$\rotatebox{90}{$\cdots$}$};

\node [bvertex, label={[font=\small] right:$u_{t-1}$}] (u3) at (1,-0.5) {};
\node [bvertex, label={[font=\small] right:$u_{t}$}] (u4) at (1,-1) {};

\draw (v2) -- (v1);
\draw (v1) -- (u0);
\draw (u0) -- (u1);
\draw (u0) -- (u2);
\draw (u0) -- (u3);
\draw (u0) -- (u4);
\end{tikzpicture}	}%
	\hfil
}
    \caption{The chair graph and the $t$-broom.}
    \label{fig:broom-graph}
\end{figure}

A 2-broom is also known as a \textit{chair} graph or {\it fork} graph, see the left graph in Figure \ref{fig:broom-graph}.  
The class of chair-free graphs is an immediate superclass of claw-free graphs and $P_4$-free graphs, both of which are polynomially $\chi$-bounded. 
Although it has been known \cite{Kierstead-Penrice1994} that the class of $t$-broom-free graphs is $\chi$-bounded, it is unknown whether this class is polynomially $\chi$-bounded. Esperet \cite{Esperet2017} conjectured that any hereditary $\chi$-bounded graph family admits a polynomial $\chi$-binding function. Very recently, extending an idea from a recent result of Carbonero, Hompe, Moore, and Spirkl \cite{CHMS2023}, Bria\'nski, Davies and Walczak \cite{BDW2022+} disproved the conjecture. For the family of chair-free graphs, Schiermeyer and Randerath \cite{Schiermeyer-Randerath2019} asked the following: 

\begin{question} [Schiermeyer and Randerath \cite{Schiermeyer-Randerath2019}]\label{que:chair}
Does there exist a polynomial $\chi$-binding function for the family of chair-free graphs?
\end{question}

There has been recent work on polynomial $\chi$-binding functions for certain subclasses of chair-free graphs. 
See \cite{CHKK2021+, KKS2021+}, where the chair graph is referred to as a \textit{fork} graph. 
In this paper, we show that the class of $t$-broom-free graphs is polynomially $\chi$-bounded. 

\begin{theorem}\label{thm:broom_binding_function}
Let $t$ be a positive integer. For $t$-broom-free graphs $G$, $\chi(G)  = o(\omega(G)^{t+1})$. 
\end{theorem}

When $t=1$, a $t$-broom-free graph is a $P_4$-free graph and, hence, perfect; so the assertion of Theorem \ref{thm:broom_binding_function} holds. When $t = 2$,  Theorem \ref{thm:broom_binding_function} answers Question \ref{que:chair} in the affirmative. Indeed, we prove a quadratic bound for the case when $t=2$, confirming a conjecture of Sivaraman mentioned in \cite{KKS2021+}.
 
\begin{theorem}\label{thm:chair_binding_function}
For all chair-free graphs $G$, $\chi(G) \leq 7\omega(G)^2$.
\end{theorem}

We remark that after the submission of our paper, Scott, Seymour and Spirkl \cite{SSS2022} extended Theorem \ref{thm:broom_binding_function} by showing that for any fixed double star $T$, the class of $T$-free graphs is polynomially $\chi$-bounded, where a \textit{double star} is a tree in which at most two vertices have degree more than one. In the case of $t$-broom-free graphs, our $\chi$-binding function is much smaller.

 Schiermeyer and Randerath \cite{Schiermeyer-Randerath2019} informally conjectured that the smallest $\chi$-binding functions for the class of chair-free graphs and the class of claw-free graphs are asymptotically the same.
 Very recently, Chudnovsky, Huang, Karthick, and Kaufmann \cite{CHKK2021+} proved that every $\{\textrm{chair}, K_{2,2}\}$-free graph $G$ satisfies that $\chi(G)\leq \ceil{\frac{3}{2}\omega(G)}$.
 Here, we consider \{$t$-broom,$K_{t,t}$\}-free graphs for $t\ge 3$ and prove the following

\begin{theorem} \label{thm:Ktt-free}
 Let $t\ge 3$ be an integer. For all \{$t$-broom, $K_{t,t}$\}-free graphs $G$, $\chi(G)=o\lp \omega(G)^{t}\rp$.
 \end{theorem}

In Section \ref{sec:structre_t-broom-free}, we obtain useful structural information on $t$-broom-free graphs, by taking an induced complete multipartite subgraph and studying subgraphs induced by vertices at certain distance from this mulitpartite graph. In Section \ref{sec:main_theorems}, we complete the proofs of Theorems~\ref{thm:broom_binding_function} and \ref{thm:chair_binding_function}. We prove Theorem \ref{thm:Ktt-free} in Section \ref{sec:Ktt}.

In the remainder of this section, we describe notation and terminology used in the paper.
For a positive integer $k$, we use $[k]$ to denote the set $\{1, \ldots, k\}$. We denote a path by a sequence of vertices in which consecutive vertices are adjacent. For a graph $G$ and $S\subseteq V(G)$, $G[S]$ denotes the subgraph of $G$ induced by $S$. We use $\alpha(G)$ to denote the independence number of $G$. 

 Let $G$ be a graph. For any $v\in V(G)$, $N_G(v)$ denotes the neighborhood of $v$ and $d_G(v)=|N_G(v)|$ is the degree of $v$ in $G$. We use $\Delta(G)$ and $\delta(G)$ to denote the maximum and minimum degree of $G$, respectively. For any positive integer $i$, let $N^i_G(S):= \{u \in V(G)\backslash S: \min\{d_G(u,v): v\in S\} =i\}$, where $d_G(u,v)$ is the distance between $u$ and $v$ in $G$. Then $N^1_G(S)=N_G(S)$ is the neighborhood of $S$ in $G$. 
Moreover, we let $N^{\geq i}_G(S):= \cup_{j=i}^{\infty} N_G^j(S)$.  When $S=\{s\}$ we write $N_G^i(s)$ instead of 
$N_G^i(\{s\})$. For any subgraph $H$ of $G$, we write $N_G^i(H)$ for $N_G^i(V(H))$ and $N_G^{\ge i}(H)$ for $N_G^{\ge i}(V(H))$.
When $G$ is clear from the context, we ignore the subscript $G$.

Let $G$ be a graph and let $S,T$ be disjoint subsets of $V(G)$. For a vertex $v\in V(G)\backslash S$, we say that $v$ is \textit{complete} to $S$ in $G$ if $vs\in E(G)$ for all $s\in S$; $v$ is \textit{anticomplete} to $S$ if $vs\notin E(G)$ for all $s\in S$; and $v$ is \textit{mixed} on $S$ if $v$ is neither complete nor anticomplete to $S$.
We say that $S$ is \textit{complete} (respectively, \textit{anticomplete}) to $T$ if all vertices in $S$ are complete (respectively, anticomplete) to $T$. 
 
\section{Structure of t-broom-free graphs}\label{sec:structre_t-broom-free}

In our proofs of the three results stated in Section \ref{sec:intro}, we work with an induced complete multipartite subgraph $Q$ of $G$ and bound the chromatic numbers of subgraphs induced by vertices 
at certain distance from $Q$. In this section, we prove a few lemmas on the structure of those subgraphs.  (Since the statements of the lemmas are somewhat technical, the interested reader may want to read Sections \ref{sec:main_theorems} and \ref{sec:Ktt} before this section.) 
The first lemma concerns $G[N^{\ge 2}(Q)]$.

\begin{lemma}\label{lem:N2(Q)_and_3beyond}
Let $t,q$ be integers with $t\ge 2$ and $q\ge 2$. Let $G$ be a $t$-broom-free connected graph and $V_1,\ldots, V_q$ be pairwise disjoint independent sets in $G$, such that $Q:=  G[\cup_{i\in [q]} V_i]$ is a complete $q$-partite graph. Suppose for every $v\in N(Q)$, $v$ is not complete to $V(Q)$. Then  $\Delta(G[ N^{\geq 2}(Q)]) <3R(t,\omega)$. 

\end{lemma}
\begin{proof}
For convenience, let $\omega=\omega(G)$. Suppose there exists
$z_k\in N^{k}(Q)$ for some $k\geq 2$ such that $|N(z_k)\cap N^{\geq 2}(Q)| \geq 3R(t,\omega)$.  Then $G$ has a path $z_k z_{k-1} \cdots z_1 z_0$, such that $z_0\in V(Q)$ and $z_{i} \in N^i(Q)$ for $i \in [k]$. Note that this path is induced. 
Since $z_1 \in N(Q)$, there exist distinct $i, j\in [q]$ such that $z_1 a_i \in E(G)$ for some $a_i \in V_i$, and $z_1 a_j \notin E(G)$ for some $a_j \in V_j$.

% {\color{blue} [Original version:]
% If $G[(N(z_k)\cap N^{\geq 2}(Q))\backslash  N(z_{k-1})]$ has an independent set of size $t$, say $T$, then $(z_k,z_{k-1}z_{k-2}, T)$\todo{ZW: $z_{k-2}$ may not be anticomplete to $T$} is an induced $t$-broom in $G$, a contradictions. So such $T$ does not exists. Then $|(N(z_k)\cap N^{\geq 2}(Q))\backslash  N(z_{k-1})|<R(t,\omega)$, which implies  $|N(z_k)\cap N^{\geq 2}(Q)\cap  N(z_{k-1})|\ge 2R(t,\omega)$. Therefore, $|(N(z_k)\cap N^{\geq 2}(Q)\cap N(z_{k-1})) \backslash N(z_{k-2})| \geq R(t,\omega)$ or $|N(z_k) \cap N^{\geq 2}(Q)\cap N(z_{k-1})\cap  N(z_{k-2})| \geq R(t,\omega)$

% In the former case, let $I_1$ be an independent set in $G[(N(z_k)\cap N^{\geq 2}(Q)\cap N(z_{k-1})) \backslash N(z_{k-2})]$ with $|I_1|=t$; then $(z_1, a_i a_j, I_1)$ (when $k= 2$) or $(z_{k-1}, z_{k-2}z_{k-3}, I_1)$ (when $k\ge 3$) contains an induced $t$-broom in $G$, a contradiction.
% In the latter case, we have $k\ge 3$ as $N(z_0)\cap N^{\ge 2}(Q)=\emptyset$. Let $I_2$ be an independent set of size $t$ in $G[N(z_k) \cap N^{\geq 2}(Q)\cap N(z_{k-1})\cap  N(z_{k-2})]$. Then $(z_1,a_ia_j,I_2)$ (when $k\ge 3$) and $(z_{k-2}, z_{k-3} z_{k-4}, I_2)$ (when $k\ge 4$) contains an induced $t$-broom in $G$,  a contradiction. 
% }

If $G[(N(z_k)\cap N^{\geq 2}(Q))\cap  N(z_{k-2})]$ has an independent set of size $t$, say $T$, then $k\geq 3$ and $(z_1, a_i a_j, T)$ (when $k=3$) or $(z_{k-2},z_{k-3} z_{k-4}, T)$ (when $k\geq 4$) is an induced $t$-broom in $G$, a contradiction. So such $T$ does not exist. Then $|(N(z_k)\cap N^{\geq 2}(Q))\cap N(z_{k-2})|<R(t,\omega)$, which implies  $|(N(z_k)\cap N^{\geq 2}(Q))\backslash  N(z_{k-2})|\ge 2R(t,\omega)$. Therefore, $|(N(z_k)\cap N^{\geq 2}(Q)\backslash N(z_{k-2})) \backslash N(z_{k-1})| \geq R(t,\omega)$ or $|(N(z_k) \cap N^{\geq 2}(Q)\backslash N(z_{k-2}))\cap  N(z_{k-1})| \geq R(t,\omega)$.

In the former case, let $I_1$ be an independent set in $G[(N(z_k)\cap N^{\geq 2}(Q)\backslash N(z_{k-2})) \backslash N(z_{k-1})]$ with $|I_1|=t$; then $(z_k, z_{k-1}z_{k-2}, I_1)$ is an induced $t$-broom in $G$, a contradiction.
In the latter case, let $I_2$ be an independent set of size $t$ in $G[(N(z_k) \cap N^{\geq 2}(Q)\backslash N(z_{k-2}))\cap  N(z_{k-1})]$. Then $(z_1,a_i a_j,I_2)$ (when $k=2$) or $(z_{k-1}, z_{k-2} z_{k-3}, I_2)$ (when $k\ge 3$) is an induced $t$-broom in $G$,  a contradiction. 
\end{proof}

%{\color{blue} If $k=2$ and $|N(z_2)\cap N^{\geq 2}(Q)|\geq 2R(t,\omega)$, then $|N(z_2)\cap N^{\geq 2}(Q)\cap N(z_1)|\geq R(t,\omega)$ or $|(N(z_2)\cap N^{\geq 2}(Q))\backslash N(z_1)|\geq R(t,\omega)$. Note that when $k\geq 3$,  $| N(z_k)\cap N(z_{k-2})|< R(t, \omega)$. For, otherwise, $G[N(z_k)\cap N(z_{k-2})]$ has an independent set $I$ of size $t$;  so $(z_{k-2}, z_{k-3} z_{k-4}, I)$ (when $k\ge 4$) or $(z_1, a_i a_j, I)$ (when $k=3$) is an induced $t$-broom in $G$, a contradiction. }

%{\color{blue} Hence $|N(z_k)\cap N^{\geq 2}(Q)\cap N(z_{k-1}) \backslash N(z_{k-2})| \geq R(t,\omega)$ or $|(N(z_k) \cap N^{\geq 2}(Q))\backslash (N(z_{k-1})\cup N(z_{k-2}))| \geq R(t,\omega)$}.
%Let $I_1$ and $I_2$ be two largest independent sets in $G$ satisfying $I_1\subseteq N(z_k)\cap N^{\geq 2}(Q)\cap N(z_{k-1}) \backslash N(z_{k-2})$ and $I_2\subseteq (N(z_k) \cap N^{\geq 2}(Q))\backslash (N(z_{k-1})\cup N(z_{k-2}))$. Note that $|I_1|\geq t$ or $|I_2|\geq t$.
%If $|I_1|\geq t$, then $(z_1, a_i a_j, I_1)$ (when $k= 2$) or $(z_{k-1}, z_{k-2}z_{k-3}, I_1)$ (when $k\ge 3$) contains an induced $t$-broom in $G$, a contradiction. If $|I_2| \geq t$, then $(z_k, z_{k-1} z_{k-2}, I_2)$ contains an induced $t$-broom in $G$,  a contradiction. 
%\end{proof}

The next lemma describes the structure of subgraphs of $G$ induced by certain subsets of $N(Q)$. 
%consisting of vertices $v$ complete or anticomplete to certain vertices of $Q$. 

\begin{lemma}\label{lem:W}
Let $t,q$ be integers with $t\ge 2$ and $q\ge 2$. Let $G$ be a $t$-broom-free graph and $V_1,\ldots, V_q$ be pairwise disjoint independent sets in $G$, such that $|V_q|= t$ and $Q:=  G[\cup_{i\in [q]} V_i]$ is a complete $q$-partite graph. Let 
\begin{itemize}
 \item $Z=\{v\in N(Q): v \mbox{ is complete to } V_q\}$,
 \item $W=\{v\in N(Q): v \mbox{ is anticomplete to } V_q\}$, and, 
 \item for each $I\subseteq \cup_{i\in [q-1]} V_i$, let $$W_I:= \{v \in W: \textrm{ $v$ is complete to $I$ and anticomplete to $\cup_{i\in [q-1]} V_i\backslash I$}\}.$$
\end{itemize}
Then the following statements hold:
\begin{itemize}
 \item [(i)] For any distinct subsets $I, I'$ of $\cup_{i\in [q-1]} V_i$, $W_I$ is anticomplete to $W_{I'}$.
 \item [(ii)] For each $z\in Z$ and for any component $X$ of $G[W]$, $z$ is complete to $V(X)$ or $z$ is anticomplete to $V(X)$.
 \item [(iii)] For any component $X$ of $G[W]$ with $\alpha(X)\ge t$, let $Z_{X}:=\{z\in Z:\textrm{ $z$ is complete to $V(X)$}\}$; then $Z_{X}$ is complete to $Z\backslash Z_{X}$.
\end{itemize}
\end{lemma}
\begin{proof}
Let $I,I'\subseteq \cup_{i\in [q-1]} V_i$ such that $I\ne I'$, and assume that $W_I$ is not anticomplete to $W_{I'}$. Then there exist $w \in W_I$ and $w'\in W_{I'}$ such that $ww'\in E(G)$. Since $I \neq I'$, we may assume that there exists $a \in I\backslash I'$. Then $(a, ww', V_q)$ is an induced $t$-broom in $G$, a contradiction. So (i) holds. 

Next, let $z\in Z$ and $w, w'\in V(X)$ such that $ww'\in E(G)$. If $zw\in E(G)$ and $zw'\not\in E(G)$, then $(z, w w', V_q)$ is an induced $t$-broom in $G$, a contradiction. Hence (ii) holds.

To prove (iii),  suppose there exist $z \in Z_{X}$ and $z'\in Z\backslash Z_{X}$ such that $zz'\notin E(G)$. By assumption, $X$ contains an independent set of size $t$, say $T$. By (ii), $z'$ is anticomplete to $V(X)$; so $z'$ is anticomplete to $T$. Choose $a \in V_q$. Now $(z, a z', T)$ is an induced $t$-broom in $G$, a contradiction. Thus we have (iii). 
\end{proof}

We now consider a specific type of complete multipartite subgraphs $Q$, as well as the vertices in $N(Q)$ that are complete to all but the last part of $Q$. We can bound the maximum degree of the subgraph of $G$ induced by such vertices.

\begin{lemma}\label{lem:B}
Let $t,q$ be integers with $t\ge 2$ and $q\ge 2$. Let $G$ be a $t$-broom-free graph and $V_1,\ldots, V_q$ be pairwise disjoint independent sets in $G$, such that $|V_i|=1$ for $i\in[q-1]$, $|V_q|=t$, and $Q:=  G[\cup_{i\in [q]} V_i]$ is a complete $q$-partite graph.
Suppose such $Q$ is chosen to maximize $q$. Let $$B:= \{ v\in N(Q): \mbox{$v$ is  complete to $V(Q)\backslash V_q$}\}.$$ Then  $\Delta(G[B]) < R(t, \omega(G))$. 
\end{lemma}
\begin{proof}
Suppose, otherwise, there exists $v\in B$ such that $d_{G[B]}(v)\geq R(t,\omega(G))$. Observe that  $\omega(G[N_{G[B]}(v)]) \leq \omega(G)-1$. Hence, $G[N_{G[B]}(v)]$ contains an independent set of size $t$, say $T$. By the definition of $B$, both $v$ and $T$ are complete to $V_j$ for all $j\in [q-1]$. Let $V_j'=V_j$ for $j\in [q-1]$, $V_q'=\{v\}$, and $V_{q+1}'=T$. Then $G[\displaystyle\cup_{j\in [q+1]}V_j']$ is an induced complete $(q+1)$-partite subgraph in $G$, contradicting the choice of $Q$.
\end{proof}

We also need to consider the vertices in $N(Q)$ that are mixed on the last part of $Q$ and not complete to some other part of 
$Q$. We can bound the number of such vertices. 

\begin{lemma}\label{lem:A}
Let $t,q$ be integers with $t\ge 2$ and $q\ge 2$. Let $G$ be a $t$-broom-free graph and $V_1,\ldots, V_q$ be pairwise disjoint independent sets in $G$, such that $|V_i|=1$ for $i\in[q-1]$, $|V_q|=t$, and $Q:=  G[\cup_{i\in [q]} V_i]$ is a complete $q$-partite graph. Let $$A:= \{ v\in N(Q): \mbox{ $v$ is mixed on $V_q$ and $v$ is not complete to $V(Q)\backslash V_q$}\}.$$ Then $|A|< t^2 \omega(G) R(t, \omega(G))$.
\end{lemma}
\begin{proof}
For each $i\in [q-1]$, let $v_i$ be the unique vertex in $V_i$. Let $v \in A$ be arbitrary. By definition there exists $j\in [q-1]$ and there exist $a_q, b_q \in V_q$, such that $v a_q \in E(G), v b_q \not\in E(G)$, and $v v_j \not\in E(G)$. If there are multiple such choices of $(a_q, b_q, j)$, we pick an arbitrary one and assign the vertex $v$ the label $(a_q, b_q, j)$. Since $q \leq \omega(G)$ and $|V_q| = t$, there are in total at most $t^2 \omega(G)$ such labels. 

Thus, if $|A| \geq t^2 \omega(G) R(t, \omega(G))$, then there exists a set $A' \subseteq A$ such that $|A'| \geq R(t, \omega(G))$ and all vertices in $A'$ receive the same label, say, $(a_q, b_q, j)$. Since all vertices in $A'$ are adjacent to $a_q$, it follows that $\omega(G[A'])<\omega(G)$. Hence, $G[A']$ must contain an independent set of size $t$, say $T$. Now $(a_q, v_j b_q, T)$ is an induced $t$-broom in $G$, a contradiction. 
\end{proof}

When $t=2$, we can substitute  Lemma~\ref{lem:A} with the following result. Its proof is the only place where we use the Strong Perfect Graph Theorem \cite{Chudnovsky-Seymour2005}: A graph is perfect if and only if it contains no odd hole or odd antihole. (Recall that a graph $G$ is {\it perfect} if $\chi(H)=\omega(H)$ for all induced subgraphs $H$ of $G$.)

\begin{lemma}\label{lem:A'}
Let $q$ be an integer with $q\ge 2$. Let $G$ be a chair-free graph and let $V_1,\ldots, V_q$ be pairwise disjoint independent sets in $G$, such that $|V_i|=1$ for $i\in [q-1]$, $|V_q|=2$, and $Q:=  G[\cup_{i\in [q]} V_i]$ is a complete $q$-partite graph. Let $V_q=\{v_q,v_q'\}$, $V_i=\{v_i\}$ for $i\in [q-1]$, and $$A:= \{ v\in N(Q): \mbox{ $v$ is mixed on $V_q$ and $v$ is not complete to $V(Q)\backslash V_q$}\}.$$  Then both $G[A\cap N(v_q)]$ and $G[A\cap N(v_q')]$ are perfect.
\end{lemma}
\begin{proof}
Let $A'=A\cap N(v_q)$ and $A''=A\cap N(v_q')$. By the definition of $A$, $|N_G(v)\cap V_q|=1$ for all $v\in A$; so $A'\cap A''=\emptyset$ and $A'\cup A''=A$. By symmetry, it suffices to prove that $G[A']$ is perfect. 

We partition $A'$ into $q-1$ pairwise disjoint sets (possibly empty) as follows. Let $A_1:=\{v\in A': \textrm{ $vv_1\notin E(G)$}\}$. Suppose for some $i\in [q-1]$,  we have defined $A_1, \ldots, A_i$. If $i=q-1$, we are done; if $i<q-1$, let $A_{i+1}:= \{v\in A'\backslash \bigcup_{j\in [i]} A_j: \textrm{ $vv_{i+1}\notin E(G)$}\}$. Hence, by the definition of $A$, $A'=\cup_{i\in [q-1]} A_i$. 

Observe that if $i\ge 2$ and $A_i\ne \emptyset$, then for all $j\in [i-1]$, $v_j$ is complete to $A_i$.
We claim that for each $i\in [q-1]$ with $A_i\ne \emptyset$, $G[A_i]$ is a clique; indeed, suppose there exist $u_1, u_2\in A_i$ such that $u_1 u_2\notin E(G)$, then $(v_q, v_i v_q', \{u_1, u_2\})$ is an induced chair in $G$, a contradiction.

Now assume for a contradiction that $G[A']$ is not perfect. Then, by the Strong Perfect Graph Theorem, $G[A']$ contains an odd hole or an odd antihole. Let $\{a_1, a_2, \ldots, a_{2k+1}\}$ be the vertex set of an odd hole or odd antihole in $G[A']$ (so $k\ge 2$). For each $i\in [2k+1]$, since $a_i\in A$,  there exists $\sigma(i)\in [q-1]$ such that $a_i\in A_{\sigma(i)}$. By symmetry, we may assume that $\sigma(1)=\min \{\sigma (i): i\in [2k+1]\}$. 

Suppose $G[\{a_1, \ldots, a_{2k+1}\}]$ is an odd hole. Without loss of generality, assume for $i,j\in [2k+1]$, $a_ia_j\in E(G)$ if $|i-j|\equiv 1$ (mod $2k-1$), and $a_ia_j\notin E(G)$ otherwise. 
Since $a_2a_{2k+1}\notin E(G)$ and $G[A_{\sigma(1)}]$ is a clique, $a_2\notin A_{\sigma(1)}$ or $a_{2k+1}\notin A_{\sigma(1)}$. By symmetry, we may assume that $a_2\notin A_{\sigma(1)}$. Hence, $\sigma(2)>\sigma(1)$ and $a_2v_{\sigma(1)}\in E(G)$.
Similarly,  $a_{2k}\notin A_{\sigma(1)}$ as $a_{2k}a_1\notin E(G)$; so $a_{2k}v_{\sigma(1)}\in E(G)$. Since $a_1a_{2k},a_2a_{2k}\notin E(G)$, $(v_{\sigma(1)}, a_2a_1, \{a_{2k},v_{q}'\})$ is an induced chair in $G$, a contradiction.

Thus, $G[\{a_1, \ldots, a_{2k+1}\}]$ is an odd antihole. Without loss of generality, we may assume that, for $i,j\in [2k+1]$, $a_ia_j\notin E(G)$ if $|i-j|\equiv 1$ (mod $2k-1$), and $a_ia_j\in E(G)$ otherwise. Note that $a_1a_3\in E(G)$ and $a_1a_2, a_2a_3\notin E(G)$. Hence $a_2\notin A_{\sigma(1)}$ and $a_2v_{\sigma(1)}\in E(G)$. If $a_3\notin A_{\sigma(1)}$, then $a_3v_{\sigma(1)}\in E(G)$; now $(v_{\sigma(1)}, a_3a_1, \{a_{2},v_q'\})$ is an induced chair in $G$, a contradiction.
So $a_3\in A_{\sigma(1)}$. Proceeding inductively, we see that $a_{2j+1}\in A_{\sigma(1)}$ for all $j\in [k]$. Thus $a_1a_{2k+1}\in E(G)$ as $G[A_{\sigma(1)}]$ is a clique. This gives a contradiction since $a_1a_{2k+1}\notin E(G)$. 
\end{proof}

We now prove a lemma about \{$t$-broom, $K_{t,t}$\}-free graphs, using a similar idea as in the proof of Lemma \ref{lem:A'} above, and show that  $\chi(G[A])$ admits a degenerate structure. A graph $H$ is said to be \textit{$d$-degenerate}, where $d$ is a positive integer, if the vertices of $H$ may be labeled as $u_1,\ldots, u_n$ such that $|N_H(u_{i})\cap \{u_{i+1}, \ldots, u_{n}\}|\le d$ for all $i\in [n-1]$. It is not hard to see that the chromatic number of a $d$-generate graph is at most $d+1$ by coloring its vertices greedily with respect to this ordering.

\begin{lemma}\label{lem:KttA}
Let $t,q$ be integers with $t\ge 2$ and $q\ge 2$. Let $G$ be a \{$t$-broom, $K_{t,t}$\}-free graph and let $V_1,\ldots, V_q$ be pairwise disjoint independent sets in $G$, such that $|V_i|=1$ for $i\in [q-1]$, $|V_q|=t$, and $Q:=  G[\cup_{i\in [q]} V_i]$ is a complete $q$-partite graph. Let $V_i=\{v_i\}$ for $i\in [q-1]$, let $a_1,a_2\in V_q$, and 
let $$A:= \{ v\in N(Q): \mbox{ $va_1\in E(G)$, $va_2\notin E(G)$, and $v$ is not complete to $V(Q)\backslash V_q$}\}.$$  Then there exists $X\subseteq A$ with $|X|\le \omega(G) (t+2) R(t-1,\omega(G))$ such that $G[A\backslash X]$ is $(2R(t,\omega(G))-1)$-degenerate. 
\end{lemma}
\begin{proof} 
We partition $A$ into $q-1$ pairwise disjoint sets (possibly empty) as follows. Let $A_1:=\{v\in A: \textrm{ $vv_1\notin E(G)$}\}$. Suppose for some $i\in [q-1]$, we have defined $A_1, \ldots, A_i$. If $i=q-1$, we are done; if $i<q-1$, let $A_{i+1}:= \{v\in A \backslash \bigcup_{j\in [i]} A_j: \textrm{ $vv_{i+1}\notin E(G)$}\}$. By the definition of $A$, we have $A=\cup_{i\in [q-1]} A_i$ and, if $i\ge 2$ and $A_i\ne \emptyset$, then $A_i$ is complete to $\{v_j\}$ for all $j\in [i-1]$. For simplicity of notation, if $x \in A_i$ for some $i \in [q-1]$, we use $A_x$ to refer to $A_i$.

Let $\omega:=\omega(G)$. Since $a_1$ is complete to $A$, $\omega(G[A])\le \omega-1$. Observe that $\alpha(G[A_i])<t$; for, if $T$ is an independent set of size $t$ in $G[A_i]$ then $(a_1, v_i a_2, T)$ is an induced $t$-broom in $G$, a contradiction. Thus, $|A_i|<  R(t,\omega)$.

% For each $i \in [q-1]$, define $A_{\geq i} := \cup_{j\geq i} A_j$ and $A_{> i} := \cup_{j> i} A_j$. 
We define a \textit{pre-order} $(A,\preceq)$ (i.e., a binary relation that is reflexive and transitive) on vertices in $A$. For any two vertices $u \in A_i$ and $v \in A_j$, let $u 
\preceq v$ if $i\leq j$. 
Moreover, we say $u\prec v$ if $i<j$. 
For any vertex $x \in A$, define $F_x = \{y \in A: x \preceq y\}$ and $F_{>x} = \{y\in A: x \prec y\}$.

Suppose for all $x \in A$, $|N(x)\cap F_x|\le 2R(t, \omega)-1$. Consider an ordering $u_1,\ldots, u_{|A|}$ of the vertices in $A$ such that $u_i \preceq u_{i+1}$ for all $i\in [|A|-1]$.  By the above assumption, $|N(u_i) \cap \{u_{i+1}, \ldots, u_{|A|}\}|\leq |N(u_i) \cap F_{u_i}| \leq 2R(t,\omega)-1$ for all $i\in [|A|-1]$. Thus, $G[A]$ is $(2R(t,\omega)-1)$-degenerate, and the assertion of Lemma \ref{lem:KttA} holds with $X= \emptyset$.

So we may assume that there exists some $x_1 \in A$ such that $|N(x_1)\cap F_{x_1}|\ge 2R(t, \omega)$. Choose a minimal such $x_1$ (with respect to $\preceq$). Let $H_1=N(x_1)\cap F_{x_1}$. Suppose for some positive integer $k\ge 2$, we have defined $H_1,\ldots, H_{k-1}$ such that $H_i\subseteq N(x_i)\cap F_{x_i}$ for all $i\in [k-1]$. If $|N_{H_{k-1}}(x) \cap F_x|\le 2R(t,\omega)-1$ for all $x\in H_{k-1}$, then we terminate this process. Otherwise pick a minimal $x_k \in H_{k-1}$ (with respect to $\preceq$) such that $|N_{H_{k-1}}(x_k) \cap F_{x_k}| \geq 2R(t,\omega)$. Then let $H_k:= N_{H_{k-1}}(x_k) \cap F_{x_k}$. %and $X_k := (H_{k-1} \cap A_{\geq j_k}) \backslash N(x_k)$. 
When the above process terminates, we obtain a sequence, say  $(x_k, H_k), k=1,\ldots,s$. Note that $\{x_k: k\in [s]\}$ induces a clique in $G$. Moreover, by definition, each $x_k$ is adjacent to $a_1 \in V_q$. Hence $s < \omega$.

Let $X := \cup_{k\in [s]} \lp(H_{k-1}\backslash H_{k})\cap F_{x_k}\rp$ where $H_0=A$. Then $A\backslash X=\bigcup_{k\in [s+1]} (H_{k-1}\backslash F_{x_k})$, where $F_{x_{s+1}}=\emptyset$. 
It suffices to show that $|(H_{k-1}\backslash H_{k})\cap F_{x_k}|\leq (t+2) R(t-1, \omega)$ for all $k\in[s]$ and that $G[A\backslash X]$ is $(2R(t,\omega)-1)$-degenerate.

First we consider $G[A\backslash X]$. For $k\in[s]$, by the minimality of $x_k$ (with respect to $\preceq$), for all $u \in H_{k-1}$ such that $u \prec x_k$, we have $|N_{H_{k-1}}(u) \cap F_{u}| \leq 2R(t,\omega)-1$. Observe that $N_{A\backslash X}(u) \cap F_u \subseteq N_{H_{k-1}}(u) \cap F_{u}$. Hence $|N_{A\backslash X}(u) \cap F_{u}| \leq 2R(t,\omega)-1$. 
By the terminating condition, $|N_{A\backslash X}(u) \cap F_u| \leq |N_{H_s}(u) \cap F_u|\leq 2R(t,\omega)-1$ for all $u\in H_s$.
Therefore, we can order the vertices in $A\backslash X$ as $u_1,\ldots, u_{|A \backslash X|}$, such that $u_i \preceq u_{i+1}$ for all $i\in [|A\backslash X|-1]$; then $|N(u_i) \cap \{u_{i+1}, \ldots, u_{|A \backslash X|}\}|\leq 2R(t,\omega)-1$ for all $i\in [|A\backslash X|-1]$. Thus, $G[A\backslash X]$ is $(2R(t,\omega)-1)$-degenerate.

To show that $|(H_{k-1}\backslash H_{k})\cap F_{x_k}| \leq (t+2)R(t-1,\omega)$ for any $k\in [s]$, we see that no vertex in $(H_{k-1}\backslash H_{k})\cap F_{x_k}$ is adjacent to $x_k$. So it suffices to bound the number of non-neighbors of $x_k$ in $H_{k-1}\cap F_{x_k}$.
Note that $F_{x_k}=A_{x_k}\cup F_{>x_k}$. Observe that $x_k$ has at most $R(t-1,\omega)-1$ non-neighbors in $A_{x_k}$; otherwise the graph induced on the vertex set of non-neighbors of $x_k$ in $A_{x_k}$ has an independent set of size $t-1$ and this implies that $\alpha(G[A_{x_k}])\ge t$, a contradiction. Thus, it suffices to bound the number of the non-neighbors of $x_k$ in $H_{k-1}\cap F_{>x_k}$, i.e., $|(H_{k-1}\cap F_{>x_k})\backslash N(X_k)|$, from above by $(t+1)R(t-1,\omega)$.

Recall that $x_k$ has at least $2R(t,\omega)$ neighbors in $H_{k-1}\cap F_{x_k}$ (as $|N_{H_{k-1}}(x_k)\cap F_{x_k}|\geq 2R(t,\omega)$).
Since $|A_{x_k}|< R(t,\omega)$, $|N_{H_{k-1}}(x_k)\cap F_{> x_k}|\geq R(t,\omega)$. Hence there exists an independent set $Y_k$ of size $t$ in $G[N_{H_{k-1}}(x_k)\cap F_{> x_k}]$. 

We claim that for each $y \in Y_k$, $y$ and $x_k$ has at most $R(t-1, \omega)$ common non-neighbors in $H_{k-1}\cap F_{> x_k}$. Otherwise there is an independent set $T$ of size $t-1$ in $G[(H_{k-1} \cap F_{>x_k})\backslash (N(x_k) \cup N(y))]$. Let $u$ be the vertex in $\{v_1, \ldots, v_{q-1}\}$ that is anticomplete to $A_{x_k}$ and complete to $F_{>x_k}$. Then $(u, y x_k, T\cup \{a_2\})$ is an induced $t$-broom in $G$, a contradiction.

Thus, each $y\in Y_k$ has at most $R(t-1,\omega)$ non-neighbors in $(H_{k-1} \cap F_{>x_k}) \backslash N(x_k)$. Suppose  $|(H_{k-1} \cap F_{>x_k}) \backslash N(x_k)| \geq (t+1) R(t-1,\omega)$. Then there exists a set $S \subseteq (H_{k-1} \cap F_{>x_k})\backslash N(x_k)$ with $|S|\geq (t+1) R(t-1,\omega)-|Y_k|R(t-1, \omega)=R(t-1,\omega)$ such that $S$ is complete to $Y_k$. Since $|S|\ge R(t-1, \omega)$,  $G[S]$ has an independent set of size $t-1$, say $S'$.  Now $(S'\cup \{x_k\}) \cup Y_k$ induces a $K_{t,t}$ in $G$, a contradiction. Hence, $|(H_{k-1} \cap F_{>x_k}) \backslash N(x_k)| < (t+1) R(t-1,\omega)$.
\end{proof}

\section{Proofs of Theorems \ref{thm:broom_binding_function} and \ref{thm:chair_binding_function}}\label{sec:main_theorems}
When $t=1$, a $t$-broom is a path on $4$ vertices. So Theorem \ref{thm:broom_binding_function} holds for $t=1$ since a $P_4$-free graph is perfect. Hence we may assume $t\geq 2$.
To prove Theorems \ref{thm:broom_binding_function} and \ref{thm:chair_binding_function}, we apply induction on $\omega(G)$. The proofs are the same, except in the case when $t=2$ we bound $\chi(G[A])$ by using Lemma~\ref{lem:A'} (instead of Lemma~\ref{lem:A}). 

Let $f(\omega)$ be a convex function satisfying
\begin{itemize}
    \item $R(t,\omega) \leq f(\omega)$ and $1\leq f(1)$, and
    \item $\lp t^2\omega R(t,\omega)+5R(t,\omega) \rp$  + $f(\omega-1)+f(1) \le f(\omega)$.
\end{itemize} 
By the generalized binomial theorem, we may choose $f(\omega)$ to be $C_t \omega^2 R(t,\omega)$ for some sufficiently large $C_t$ depending on $t$. Hence, using the upper bound of $R(t,\omega)$, $f(\omega)$ may be chosen such that $f(\omega) = o(\omega^{t+1})$.  

We will show that $\chi(G)\le f(\omega(G))$.
Note that the assertions of Theorems \ref{thm:broom_binding_function} and \ref{thm:chair_binding_function} are clearly true when $\omega(G) = 1$. Hence, let $G$ be a $t$-broom-free graph with $\omega(G) = \omega \geq 2$, and assume that for all $t$-broom-free graphs $H$ with $\omega(H)<\omega$, we have $\chi(H)\leq f(\omega(H))$. 

We choose pairwise disjoint independent sets $V_1, \ldots, V_q$ in $G$, such that 
\begin{itemize}
  \item [(1)] $|V_q|=t$ and $|V_i|=1$ for $i\in [q-1]$,  
  \item [(2)] $Q:=G[\displaystyle\cup_{i\in [q]}V_i]$ is a complete $q$-partite graph, and
  \item [(3)] subject to (1) and (2), $q$ is maximum.
\end{itemize}
Note that $q \geq 2$, otherwise $G$ is $K_{1,t}$-free; so $\Delta(G) < R(t,\omega)$ (hence $\chi(G)\le R(t,\omega)\leq f(\omega)$) and we are done. 
Clearly,  $q\le \omega$.  

We study the structure of $G$ by partitioning $G$ into several vertex disjoint subgraphs and bounding the chromatic number of each subgraph. 
We partition $N(Q)$ as follows: 
\begin{itemize}
\item $A:= \{ v\in N(Q): \textrm{ $v$ is mixed on $V_q$ and $v$ is not complete to $V(Q)\backslash V_q$}\}$.

\item $B:= \{ v\in N(Q): \textrm{ $v$ is mixed on $V_q$ and $v$ is  complete to $V(Q)\backslash V_q$}\}$.

\item $C:=N(Q)\backslash (A\cup B)$.
\end{itemize}
By the maximality of $q$, no vertex in $N(Q)$ is complete to $V(Q)$. Thus, for any $v\in C$, either $v$ is anticomplete to $V_q$, or $v$ is complete to $V_q$ and not complete to $V(Q)\backslash V_q$.
Note that $$V(G)=V(Q)\cup N(Q)\cup N^{\geq 2}(Q) \mbox{ and } N(Q)=A\cup B\cup  C.$$ 

Since there is no edge between $Q$ and $N^{\geq 2}(Q)$, we can color $Q$ and $G[N^{\geq 2}(Q)]$ with at most $\max \{\chi(Q), \chi(G[N^{\geq 2}(Q)])\}$ colors.
Since $\chi(Q)=q$ and $\chi(G[N^{\geq 2}(Q)]) \leq 3R(t,\omega)$ by Lemma \ref{lem:N2(Q)_and_3beyond}, we then obtain the following claim.
\begin{claim}\label{cl:chi(G)}
$\chi(G) \leq \max \{q, 3R(t,\omega)\} + \chi(G[A]) + \chi(G[B]) + \chi(G[C]).$
\end{claim}

Since $\Delta(G[B]) < R(t, \omega)$ (by Lemma \ref{lem:B}), we have $\chi(G[B])\le R(t,\omega)$. By Lemmas \ref{lem:A} and \ref{lem:A'}, $\chi(G[A])\le t^2 \omega R(t, \omega)$ (when $t\ge 3$) and 
$\chi(G[A])\le 2\omega$ (when $t=2$). 
Thus, in view of Claim \ref{cl:chi(G)}, we need to bound $\chi(G[C])$. 

Let $Z = N(V_q)\cap C$ and $W=C\backslash Z$; so $Z$ is complete to $V_q$ and $W$ is anticomplete to $V_q$. We consider $G[W]$ first. For any component $X$ of $G[W]$, if $I$ is the set of all its neighbors in $V(Q)\backslash V_q$ (note that $ I \subseteq \cup_{i\in[q-1]} V_i$ and $I\ne \emptyset$) then by (i) of Lemma \ref{lem:W}, $V(X)$ is complete to $I$ and anticomplete to $\cup_{i\in[q-1]}V_i \backslash I$. It follows that any component of $G[W]$ has clique number at most $\omega-1$. Thus, each component of $G[W]$ with independent number at most $t-1$ has at most $R(t, \omega)-1$ vertices. 
Let $X_0$ denote the union of all components of $G[W]$ with independence number at most $t-1$; so $\chi(X_0)\le R(t, \omega)$.

\begin{claim}\label{cl:DminusW0}
$\chi(G[C]- X_0])\leq f(\omega-1) + f(1)$. 
\end{claim}
\begin{proof} For any component $X$ of $G[W]-X_0$, let $$Z_{X}:=\{z\in Z:\textrm{ $z$ is complete to $V(X)$}\}.$$
Let $S_1, S_2,  \ldots, S_p$ be the sets that form the smallest partition of $Z$ which refines all the bipartitions $Z_{X}, Z\backslash Z_{X}$ of $Z$ for all components $X$ of $G[W]-X_0$.
By (iii) of Lemma \ref{lem:W},  $S_i$ is complete to $S_j$ for all distinct $i, j\in [p]$. Let $\omega_i:=\omega(G[S_i])$ for $i\in [p]$. For each component $X$ of $G[W]-X_0$, let $F_{X}:= \{k \in [p]: \textrm{ $S_k$ is anticomplete to $V(X)$}\}$. Then by (ii) of Lemma \ref{lem:W}, $[p]\backslash F_{X}=\{k\in [p]: S_k \mbox{ is complete to } V(X)\}$. 

Observe that $\omega(G[Z]) < \omega$ and $\omega(X) < \omega$ for every component $X$ of $G[W]-X_0$. Hence the induction hypothesis applies to all subgraphs of $G[Z]$ and all $G[X]$.
We now describe, inductively, a coloring of $G[C]-X_0$ and show it uses at most $f(\omega-1) +f(1)$ colors.
\begin{itemize}
    \item For each $i \in [p]$, inductively color the vertices of $G[S_i]$ with colors from a set $R_i$, where $|R_i| \leq f(\omega_i)$. We choose $R_i$, $i\in [p]$, such that $R_i\cap R_j=\emptyset$ whenever $i\ne j$.
    
    \item  Let $R$ be a fixed set of colors, such that $R$ is disjoint from $\cup_{i\in [p]}R_i$ and $|R|=\max\{ f(\omega(X)) - \sum_{k\in F_{X}} |R_k| ,0\}$, where the maximum is taken over all components $X$ of $G[W]-X_0$.

    \item For each component $X$ of $G[W]-X_0$, it follows from induction that $\chi(X)\le f(\omega(X))$. Note that the vertex set of $X$ is anticomplete to $S_k$ if $k\in F_X$. Thus we can use the colors used on $\cup _{k\in F_X} G[S_k]$ (i.e., colors in $\cup_{k\in F_X} R_k$) to color $X$ first. If $\chi(X)\le \sum_{k\in F_{X}} |R_k|$, we are done. Otherwise we use at most $f(\omega(X)) - \sum_{k\in F_{X}} |R_k|$ colors from $R$ to color $X$. Hence for each component $X$ of $G[W]-X_0$, we assign the vertices of $X$ with colors from $\displaystyle\cup_{k\in F_{X}}R_k$ and, if needed, some additional colors from $R$.  
\end{itemize}

Therefore, we have $\chi(G[C]- X_0)=\chi(G[Z]\cup (G[W] -X_0))\leq \sum_{i\in [p]} |R_i| +|R|$.
Since $S_i$ is complete to $S_j$ for all  distinct $i,j \in [p]$ and $Z$ is complete to $V_q$, we have $\dss_{i=1}^p \omega_i \leq \omega-1$. Moreover, for each component $X$ in $G[W]-X_0$,  $V(X)$ is complete to $S_k$ if $k\in [p]\backslash F_{X}$,; therefore,
$\omega(X) + \dss_{k \in [p]\backslash F_{X}} \omega_k \leq \omega.$ Thus, by the convexity of the function $f$, 
we have $$\dss_{i=1}^p f(\omega_i) \leq f(\omega-1),$$ and 
$$f(\omega(X)) + \dss_{k \in [p]\backslash F_{X}} f(\omega_k) \leq f(\omega-1)+f(1).$$
Thus,
\begin{align*}
    \chi(G[C]- W_0) & \leq |R|+ \dss_{i=1}^p |R_i| \notag \\
    & = \max_{\substack{X: \text{ a component} \\ \text{of $G[W]-X_0$ }}} \left\{ f(\omega(X)) - \dss_{i\in F_{X}} |R_i|,0\right\} + \dss_{i=1}^p |R_i| \notag\\
    & = \max_{\substack{X: \text{ a component} \\ \text{of $G[W]-X_0$ }}} \left\{ f(\omega(X)) + \dss_{i\in [p]\backslash F_{X}} |R_i|,\dss_{i=1}^p |R_i|\right\} \notag\\
    & \leq \max_{\substack{X: \text{ a component} \\ \text{of $G[W]-X_0$ }}} \left\{ f(\omega(X)+ \dss_{i \in [p]\backslash F_{X}} f(\omega_i), \dss_{i=1}^p f(\omega_i)\right\}. \notag\\
    & \leq f(\omega-1) +f(1) \label{eq:N1Q}.
\end{align*}
This completes the proof of Claim \ref{cl:DminusW0}
\end{proof}

We can now bound $\chi(G)$. 
Since $\chi(X_0) \le R(t, \omega)$, it follows from Claim \ref{cl:DminusW0} that 
$$\chi(G[C]) \leq  \chi(G[C]- X_0]) + \chi(X_0)\le f(\omega-1) +f(1) + R(t, \omega).$$
%Note that there exists a constant $c'(t) > 0$ such that
%$$3R(t,\omega) + t^2\omega R(t,\omega)+R(t,\omega)+R(t,\omega+1)\leq c'(t) \omega R(t, \omega).$$
Hence, by Claim \ref{cl:chi(G)}, Lemma \ref{lem:B} and Lemma \ref{lem:A}, we have 
\begin{align*}
    \chi(G) & \leq \max \{\omega, 3R(t,\omega)\} + t^2\omega R(t,\omega)+R(t,\omega)+f(w-1)+f(1)+R(t,\omega)\\
            &\leq t^2 \omega R(t,\omega) + 5R(t,\omega) + f(\omega-1) + f(1) \\
            &\leq f(\omega),
\end{align*}
by our choice of $f(\omega)$. This proves Theorem \ref{thm:broom_binding_function}.

\medskip
When $t=2$, we have $R(2,\omega)=\omega$, and $\chi(G[A])\le 2\omega(G)$ (by Lemma \ref{lem:A'}). Thus, using those bounds in the above inequality, we obtain 
\begin{align*}
    \chi(G)       \leq f(\omega-1)+f(1)+7\omega.
\end{align*}
By choosing $f(\omega)=7\omega^2$, we see that $f(\omega-1)+f(1)+7\omega\le f(\omega)$. Hence, $\chi(G)\le 7\omega^2$, completing the proof of Theorem \ref{thm:chair_binding_function}. 

\section{Proof of Theorem \ref{thm:Ktt-free}}\label{sec:Ktt}
\setcounter{claim}{0}
 
Let $g(\omega)$ be a convex function satisfying
\begin{itemize}
    \item $1\leq g(1)$,
    \item $\omega + R(t,\omega) + (t+2){t^2}\omega R(t-1,\omega)+(2{ t^2}+4) R(t,\omega)\leq g(\omega)$, and
    \item $g(\omega-1) + \omega + R(t,\omega) \leq g(\omega).$
\end{itemize}
Similar to the proof in Section \ref{sec:main_theorems}, by the generalized Binomial Theorem, $g(\omega)$ can be chosen such that $g(\omega) = C_t\lp \omega R(t,\omega) + \omega R(t-1,\omega)\rp$ for some large constant $C_t$ depending only on $t$. Hence we may choose $g(\omega)$ such that $g(\omega) = o(\omega^{t})$. 
 
We will show that $\chi(G)\le g(\omega(G))$ by applying induction on $\omega(G)$. It is clear that the assertion of the theorem holds when $\omega(G)=1$. Let $G$ be a \{$t$-broom, $K_{t,t}$\}-free graph with $\omega(G)=\omega \geq 2$ and, for all \{$t$-broom, $K_{t,t}$\}-free graphs $H$ with $\omega(H)<\omega$, we have $\chi(H)\leq g(\omega(H))$.

We choose pairwise disjoint independent sets $V_1, \ldots, V_q$ in $G$, such that 
\begin{itemize}
  \item [(1)] $|V_q|=t$ and $|V_i|=1$ for $i\in [q-1]$,  
  \item [(2)] $Q:=G[\displaystyle\cup_{i\in [q]}V_i]$ is a complete $q$-partite graph, and
  \item [(3)] subject to (1) and (2), $q$ is maximum. 
\end{itemize}
Such $Q$ with $q\geq 2$ must exist, otherwise $G$ is $K_{1,t}$-free and hence $\Delta(G) < R(t,\omega)$ and we are done. 
Clearly,  $2\le q\le \omega$. Let $V_i=\{v_i\}$ for $i\in [q-1]$.
We partition $N(Q)$ as follows. 
\begin{itemize}
\item $A:= \{ v\in N(Q): \textrm{ $v$ is mixed on $V_q$ and $v$ is not complete to $V(Q)\backslash V_q$}\}$.

\item $B:= \{ v\in N(Q): \textrm{ $v$ is mixed on $V_q$ and $v$ is  complete to $V(Q)\backslash V_q$}\}$.

\item $C:=N(Q)\backslash (A\cup B)$.
\end{itemize}
Thus, for each $v\in C$,  either $v$ is complete to $V_q$ or $v$ is anticomplete to $V_q$.  Let $Z = N(V_q)\cap C$ and $W=C\backslash Z$; so $Z=\{v\in C: v \mbox{ is complete to } V_q\}$ and $W:=\{v\in C: v \mbox{  is anticomplete to } V_q\}$. Note that  $N(Q)$ is the disjoint union of $A, B, Z, W$. 

\begin{claim} \label{cl:ABZ}
$N(W)\cap N^2(Q)=\emptyset$, $\chi(G[A]) \le { t^2}(2R(t,\omega)+(t+2)\omega R(t-1, \omega))$, $\Delta(G[B]) < R(t, \omega)$, and
$|Z| \le R(t,\omega)$.
\end{claim}
\begin{proof} 
Suppose there exists $wy\in E(G)$ with $w\in W$ and $y\in N^2(Q)$. Choose $i\in [q-1]$ such that $wv_i\in E(G)$. Now $(v_i,wy,V_q)$ is an induced $t$-broom in $G$, a contradiction. Hence, $N(W)\cap N^2(Q)=\emptyset$.

By Lemma \ref{lem:B}, we have $\Delta(G[B]) < R(t, \omega)$; hence $\chi(G[B])\le R(t,\omega)$. 
For $Z$, recall that $Z$ is complete to $V_q$. Hence $\omega(G[Z])\leq \omega-1$. Moreover, since $|V_q|= t$ and $G$ is $K_{t,t}$-free, we have that $\alpha(G[Z]) < t$. So $|Z| \leq R(t,\omega)$.

It remains to bound  $\chi(G[A])$. 
Write $V_q = \{a_1, a_2, \ldots, a_t\}$. 
Since $A$ is mixed on $V_q$, for each $v\in A$, there exist $i,j \in [t]$ such that  $v a_i \in E(G)$ and $v a_j \not\in E(G)$.
Let $A^{(i,j)} = \{v \in A: v a_i \in E(G), \textrm{ and } v a_j \not\in E(G)\}$. 
%It suffices to show that $\chi(G[A^{(i,j)}])  = O(\omega R(t-1, \omega))$ for every pair $i,j\in [t]$. Without loss of generality, we show that $\chi(G[A^{(1,2)}]) =O(\omega R(t-1, \omega))$. 
By Lemma \ref{lem:KttA}, there exists $X\subseteq A^{(i,j)}$ such that $|X|\le (t+2)\omega R(t-1,\omega)$ and 
$G[A^{(i,j)}\backslash X]$ is $(2R(t,\omega)-1)$-degenerate. Hence, 
$\chi(G[A^{(i,j)}]) \leq 2R(t,\omega) + |X| = 2R(t,\omega)+(t+2)\omega R(t-1, \omega).$ Thus, 
$\chi(G[A])\le  t^2 (2R(t,\omega)+ (t+2)\omega R(t-1, \omega)).$
\end{proof}

Next we consider $G[W]$. It follow from (i) of Lemma~\ref{lem:W} that, for any component $X$ of $G[W]$, $V(X)$ is complete to its neighborhood in $V(Q)\backslash V_q$. Let $X_0$ denote the union of all components of $G[W]$ with chromatic number at most $3R(t, \omega)$. By the definition of $X_0$ and the fact that every component of $G[W]$ has clique number at most $\omega-1$, we have the following claim.
\begin{claim}\label{cl:WW0}
$\chi(X_0) \leq 3R(t, \omega)$ and $\chi(G[W]-X_0)\leq g(\omega-1)$.
\end{claim}

%Now we consider the edges of $G$ between $A\cup B$ and $W\backslash W_0$.

\begin{claim}\label{cl:ABW0}
$A\cup B$ is anticomplete to $W\backslash V(X_0)$. 
\end{claim}
\begin{proof}
For any component $X$ in $G[W]-X_0$, $\chi(X)>3R(t,\omega)$ by the definition of $X_0$. This implies that $|V(X)|\ge \chi(X)>3R(t,\omega)$; hence $X$ contains an independent set of size $t$.

 We claim  that, for any distinct components $X_1,X_2$ of $G[W]-X_0$, $N(X_1)\cap V(Q) \subseteq N(X_2)\cap V(Q)$ or $N(X_2) \cap V(Q) \subseteq N(X_1)\cap V(Q)$. 
 For, suppose there exist distinct $u_1,u_2\in V(Q)$ such that $u_1\in (N(X_1)\backslash N(X_2))\cap V(Q)$ and $u_2\in  (N(X_2)\backslash N(X_1))\cap V(Q)$. We know that $X_2$ has an independent set of size $t$, say $T_2$. Let $x_1$ be a vertex of $X_1$; then  $(u_2, u_1x_1, T_2)$ is an induced $t$-broom in $G$, a contradiction.
 
 Thus, we choose a component $X$ of $G[W]-X_0$ such that $N(X)\cap V(Q)$ is maximal. Observe that $N(X)\cap V(Q)=N(X)\cap (V(Q)\backslash V_q)$ which is a proper subset of $V(Q)\backslash V_q$; otherwise $V(X)$ is complete to  $V(Q)\backslash V_q$ and $\Delta(X)<R(t,\omega)$ by Lemma~\ref{lem:B}, contradicting that $\chi(X)>3R(t,\omega)$. So there exists $j\in [q-1]$ such that $v_j\notin N(X)$. Hence by the choice of $X$, $v_j \notin N(X')$ for any component $X'$ of  $G[W]-X_0$. This implies that $v_j$ is anticomplete to $W\backslash V(X_0)$. 
 
 Now suppose there exists a vertex $a$ in $A\cap B$ such that $a$ is not anticomplete to $W\backslash V(X_0)$. Then there exists a component $X$ of $G[W]-X_0$ and $w\in V(X)$ such that $aw\in E(G)$. Since $a$ is mixed on $V_q$,  we may assume without loss of generality that $a_1,a_2\in V_q$ such that $aa_1\in E(G)$ and $aa_2\notin E(G)$.  Note that $\chi(X)>3R(t,\omega)$ and recall that $v_j$ is anticomplete to $V(X)$.

If $av_j\in E(G)$ then $G':=G[\{v_j,a_2,a\}\cup V(X)]$ is a $t$-broom-free graph and $\omega(G')\leq \omega(G)=\omega$. Note that $Q':=G'[\{v_j,a_2\}]$ is a complete bipartite subgraph of $G'$, $N_{G'}(Q')=\{a\}$, and $G'[N^{\geq 2}_{G'}(Q')]=X$. Since $a$ is not complete to $V(Q')$ in $G'$, we may apply Lemma~\ref{lem:N2(Q)_and_3beyond} and conclude that $\Delta(G'[N^{\geq 2}_{G'}(Q')])<3R(t,\omega)$. Hence, $\chi(G'[N_{G'}^{\geq 2}(Q')])\le 3R(t,\omega)$, a contradiction as $X=G'[N^{\geq 2}_{G'}(Q')]$ and $\chi(X)>3R(t,\omega)$. 

Now assume that $av_j\notin E(G)$. Then $G'':=G[\{v_j,a_1,a\}\cup V(X)]$ is a $t$-broom-free graph and $Q'':=G''[\{v_j,a_1\}]$ is a complete bipartite subgraph of $G''$,  $N_{G''}(Q'')=\{a\}$, and $X= G''[N^{\geq 2}_{G''}(Q'')]$. Since $a$ is not complete to $V(Q'')$, we may apply Lemma~\ref{lem:N2(Q)_and_3beyond} and conclude that $\Delta(G''[N^{\geq 2}_{G''}(Q'')])<3R(t,\omega)$. Hence, $\chi(G''[N_{G''}^{\geq 2}(Q'')])\le 3R(t,\omega)$, a contradiction as $X=G''[N^{\geq 2}_{G''}(Q'')]$ and $\chi(X)>3R(t,\omega)$.
\end{proof}

Note that $V(G)=V(Q)\cup N(Q)\cup N^{\ge 2}(Q)$ and $N(Q)=A\cup B\cup Z\cup V(X_0) \cup (W\backslash V(X_0))$ and $\chi(Q)=q$. Also note that $W\backslash V(X_0)$ is anticomplete to $A\cup B\cup N^{\geq 2}(Q)$ (by Claims \ref{cl:ABZ} and \ref{cl:ABW0}), $V(X_0)$ is anticompete to $N^{\geq 2}(Q)$ (by Claim \ref{cl:ABZ}), and $W\backslash V(X_0)$ is anticomplete to $V(X_0)$ (by definition). Thus, we have 
$$\chi(G) \leq  q+ |Z|+
 \max \left\{ \chi(G[W\backslash V(X_0)]),\chi(G[A])+\chi(G[B])+\max \{\chi(X_0), \chi(G[N^{\geq 2}(Q)])\} \right\}.$$

By the maximality of $q$, no vertex in $N(Q)$ is complete to $V_j$ for all $j\in [q]$. Thus, by Lemma \ref{lem:N2(Q)_and_3beyond}, $\chi(G[ N^{\geq 2}(Q)]) \le  3R(t,\omega)$. 
Therefore, 
$$\chi(G) \leq  \omega+ R(t,\omega) +
 \max \{ g(\omega-1),(t+2){t^2}\omega R(t-1,\omega)+(2{ t^2}+4) R(t,\omega)\}.$$
Hence, by the choice of $g(\omega)$, we have $\chi(G)\le g(\omega)$, completing the proof of Theorem \ref{thm:Ktt-free}. \\

{\bf Acknowledgement} We would like to sincerely thank the anonymous referees for their valuable comments and suggestions that greatly improved the manuscript.

\end{document}